\documentclass[11pt,a4paper,reqno,english]{ijocta}

\usepackage[dvips]{graphicx}
\usepackage[margin=2cm]{geometry}
\usepackage[font=small,labelfont=bf,tableposition=top]{caption}
\usepackage[font=footnotesize]{subcaption}
\usepackage{multicol}
\usepackage{babel}
\usepackage{amssymb,amsmath,amsfonts}
\usepackage{multirow,url}
\usepackage{blindtext} 

\newtheorem{theorem}{Theorem}

\newtheorem{algorithm}{Algorithm}

\newtheorem{corollary}{Corollary}

\newtheorem{definition}{Definition}

\newtheorem{lemma}{Lemma}

\newtheorem{remark}{Remark}

\begin{document}

\title[\small{\textit{A novel hybrid method for equilibrium problems and fixed point problems}}]{\textbf{A novel hybrid method for equilibrium problems and fixed point problems}}
\author[\small{\textit{D. V. Hieu\\/Vol.?, No.?, pp.?-? (2015) \copyright JNAO}}]{\normalsize Dang Van Hieu$^a$\vspace{0.58cm} %please don't remove \vspace{.3cm}
\\ \small
Department of Mathematics, Vietnam National University, Vietnam \\
334 Nguyen Trai Street, Thanh Xuan, Hanoi, Vietnam.\\
%Email: dv.hieu83@gmail.com \vspace{0.1cm} \\
% $^b$Department of Mathematics Bal\i kesir University, Turkey \\
% Email: davci@balikesir.edu.tr \vspace{0.4cm} \\
%\textit{(Received xx $xx, 2015$; in final form xx $xx, 2015$)}
}
  
\thanks{$^a$Corresponding Author. Email: dv.hieu83@gmail.com}

\begin{abstract}
The paper proposes a novel hybrid method for solving equilibrium problems and fixed point problems. By constructing 
specially cutting-halfspaces, in this algorithm, only an optimization program is solved at each iteration without the extra-steps as in 
some previously known methods. The strongly convergence theorem is established and some numerical examples are 
presented to illustrate its convergence.

\vspace{7pt}

\vspace{2pt}
\end{abstract}

\maketitle

\section{Introduction}\label{intro}
Let $H$ be a real Hilbert space and $C$ be a nonempty closed convex subset of $H$. Let $f:C\times C\to \Re$ be a bifunction. 
The equilibrium problem (EP) for the bifunction $f$ on $C$ is to find $x^*\in C$ such that
\begin{equation}\label{EP}
f(x^*,y)\ge 0,~\forall y\in C.
\end{equation}
The solution set of EP for $f$ on $C$ is denoted  by $EP(f,C)$. EP is also well-known as Ky Fan inequality \cite{KF1972}. 
It is very general in the sense that it includes many mathematical models as: variational inequalities, 
optimization problems, fixed point problems, Nash equilirium point problems, complementarity problems, operator equations, see, 
for instance \cite{BO1994,CH2005,DGM2003,MO1992}. A special case for the bifunction 
$f(x,y)=\left\langle A(x),y-x\right\rangle$, where $A:C\to H$ is a operator, then EP becomes the variational inequality 
problem: find $x^*\in C$ such that 
\begin{equation}\label{VIP}
\left\langle A(x^*),y-x^*\right\rangle\ge 0,~\forall y\in C.
\end{equation}
% Some algorithms for EPs 
% can be found in \cite{A2011,A2012,A2013,AH2014b,BO1994,CH2005,DGM2003,DHM2014,H2015,H2015a,MPPP2012,SVH2013,TT2007}. 
One of the most popular methods for solving EPs is the proximal point 
method (PPM) in which solution approximations are based on the resolvent \cite{CH2005} of equilibrium bifunctions 
\cite{AH2014b,CH2005,DHM2014,H2015,K2000,MPPP2012,SVH2013,TT2007}.
PPM was first introduced by Martinet \cite{M1970} for variational inequalities, and then it was extended for finding 
zero points of maximal monotone operators by Rockafellar \cite{R1976}. This method was further extended to Ky 
Fan inequalities by Konnov \cite{K2000} for monotone or weakly monotone bifunctions.

In recent years, the Korpelevich's extragradient method \cite{K1976} for variational inequalities has been extended and widely studied to EPs, 
see, for instance \cite{A2011,A2012,A2013,DHM2014,H2015a,VSH2013,QMH2008,SVH2013,SNV2012}. In 2008, Quoc et al. \cite{QMH2008} introduced 
the following algorithm for EPs in Euclidean space 

\begin{equation}\label{QMH2008}
\left \{
\begin{array}{ll}
y_{n}=  \underset{y\in C}{\rm argmin} \{ \lambda f(x_n, y) +\frac{1}{2}||x_n-y||^2\},\\
x_{n+1}=  \underset{y\in C}{\rm argmin} \{ \lambda f(y_n, y) +\frac{1}{2}||x_n-y||^2\}.
\end{array}
\right.
\end{equation}
Then, the algorithm was further extended to Hilbert spaces \cite{A2011,A2012,A2013,H2015a,VSH2013,SVH2013,SNV2012}. 
In 2013, for finding a common element of the solution set $EP(f,C)$ and the fixed point set $F(S)$ of a nonexpansive mapping $S:C\to C$, 
Anh \cite{A2013} introduced the following hybrid algorithm in Hilbert spaces
\begin{equation}\label{A2013}
\left \{
\begin{array}{ll}
y_{n}=  \underset{y\in C}{\rm argmin} \{ \lambda f(x_n, y) +\frac{1}{2}||x_n-y||^2\},\\
z_{n}=  \underset{y\in C}{\rm argmin} \{ \lambda f(y_n, y) +\frac{1}{2}||x_n-y||^2\},\\
t_n=\alpha_n x_n+(1-\alpha_n)S(w_n),\\
C_n=\left\{z\in C: ||t_{n} - z||\leq ||x_n-z||\right\},\\
Q_n=\left\{z\in C: \left\langle x_0-x_n,z-x_n\right\rangle\le 0\right\},\\
x_{n+1}=P_{C_n\cap Q_n}(x_0).
\end{array}
\right.
\end{equation}
And the author proved the sequence $\left\{x_n\right\}$ generated by (\ref{A2013}) converges strongly to $P_{EP(f,C)\cap F(S)}(x_0)$ 
under the hypothesises of the pseudomonotonicity and Lipschitz-type continuity of bifunction $f$. In the special case, EP (\ref{EP}) is 
VIP (\ref{VIP}) then (\ref{A2013}) becomes the following hybrid algorithm
\begin{equation}\label{NT2006a}
\left \{
\begin{array}{ll}
y_{n}=  P_C(x_n-\lambda A(x_n)),\\
z_{n}=  P_C(x_n-\lambda A(y_n)),\\
t_n=\alpha_n x_n+(1-\alpha_n)S(w_n),\\
C_n=\left\{z\in C: ||t_{n} - z||\leq ||x_n-z||\right\},\\
Q_n=\left\{z\in C: \left\langle x_0-x_n,z-x_n\right\rangle\le 0\right\},\\
x_{n+1}=P_{C_n\cap Q_n}(x_0),
\end{array}
\right.
\end{equation}
which was proposed by Nadezhkina and Takahashi in 2006 for variational inequalities and nonexpansive mappings \cite{NT2006a}.
Some modifications of (\ref{NT2006a}) were known as the subgradient extragradient algorithm \cite{CGR2011} in which the second 
projection onto $C$ was replaced by one onto a halfspace in Hilbert spaces and the self-adaptive subgradient extragradient algorithm 
\cite{G2015} in Euclidean spaces.

In the algorithms (\ref{QMH2008}), (\ref{A2013}) and others \cite{A2011,A2012,A2013,H2015a,VSH2013,SVH2013,SNV2012}, 
we see that two optimization programs onto the feasible set $C$ must be solved at each iteration. This seems to be costly and might seriously affect the efficiency of used methods if 
the structure of the constrained set $C$ is complex. To avoid 
solving the second optimization program in (\ref{A2013}) and the condition of Lipschitz-type continuity posed on the bifunction $f$, 
Strodiot et al. \cite{SNV2012} 
replaced it by the Armijo linesearch technique.
\begin{equation}\label{SNV2012}
\left \{
\begin{array}{ll}
y_{n}=  \underset{y\in C}{\rm argmin} \{ \lambda f(x_n, y) +\frac{1}{2}||x_n-y||^2\},\\
\mbox{Find}~$m$~ \mbox{the smallest nonnegative integer such that}\\
\left \{
\begin{array}{ll}
z_{n,m}=(1-\gamma^m)x_n+\gamma^m y_n,\\
f(z_{n,m},x_n)-f(z_{n,m},y_n)\ge \frac{\alpha}{2\lambda_n}||x_n-y_n||^2,\\
\end{array}
\right.\\
w_n=P_C(x_n-\sigma_n g_n),\\
t_n=\alpha_n x_n+(1-\alpha_n)\left(\beta_n w_n+(1-\beta_n)S(w_n)\right),\\
C_{n+1}=\left\{z\in C_n: ||t_{n} - z||\leq ||x_n-z||\right\},\\
x_{n+1}=P_{C_{n+1}}(x_0),
\end{array}
\right.
\end{equation}
where $g_n\in \partial_2 f(z_{n,m},x_n)$, $\sigma_n=f(z_{n,m},x_n)/||g_n||^2$ if $x_n\ne y_n$ and $\sigma_n=0$ otherwise.
 However, we still have to solve an optimization program, find an optimization direction, and compute 
a projection onto $C$ at each iteration. Moreover, the sets $C_{n+1}$ is not easy to compute. 

In this paper, motivated and inspired by the results in \cite{A2013,MS2015,SNV2012,VSH2013}, 
we introduce the following hybrid algorithm for finding a solution 
of an EP for a monotone and Lipschitz-type continuous bifunction $f$ which is also a fixed point of a 
nonexpansive mapping $S:C\to C$.
\begin{algorithm}[The hybrid algorithm without the extra-steps]\label{H2015}
$$
\left \{
\begin{array}{ll}
y_{n+1}=  \underset{y\in C}{\rm argmin} \{ \lambda f(y_n, y) +\frac{1}{2}||x_n-y||^2\},\\
z_{n+1}=\alpha_ny_{n+1}+(1-\alpha_n)Sy_{n+1},\\
x_{n+1}=P_{\Omega_n}(x_0),
\end{array}
\right.
$$
\end{algorithm}
\noindent where $\Omega_n=C_n\cap Q_n$ and $C_n,~Q_n$ are two specially constructed half-spaces (see Algorithm $\ref{algor1}$ in Section $\ref{main}$ below). 
In this algorithm we do not use the PPM and the resolvent of equilibrium bifunction \cite{CH2005,K2000}. Contrary 
to the extended extragradient methods and the Armijo linesearch methods \cite{A2011,A2012,A2013,DHM2014,H2015a,VSH2013,QMH2008,SNV2012,SVH2013}, in our proposed algorithm, only an optimization program is solved at 
each iterative step without the extra-steps. Note that $C_n$ and $Q_n$ in the processes (\ref{A2013}), (\ref{NT2006a}) and (\ref{SNV2012}) are constructed 
on the constrained set $C$ while $\Omega_n$ in 
Algorithm $\ref{H2015}$ is the intersection of two halfspaces, so $x_{n+1}$ can be expressed by an explicit formula, see, for instance \cite{CH2005,SS2000}. 
%%%%%%%%%%%%%%%%%%%%%%%%%%%%%%%%%%%%%%%%%%%%%%%%%%

The remainders of the paper is organized as follows: Section \ref{pre} reviews several definitions and results for further use. Section \ref{main} deals with 
analyzing the convergence of the proposed algorithm. Finally, in Section \ref{Numer} we give two numerical examples to illustrate the convergence of the algorithm. 
\section{Preliminaries}\label{pre}
%%%%%%%%%%%%%%%%%%%%%%%%%%%%%%%%%%%%%
In this section, we recall some definitions and preliminary results used in this paper. A mapping $S:C\to H$ is said to be nonexpansive if $||S(x)-S(y)||\le ||x-y||$ for 
all $x,y\in C$. The set of fixed points of $S$ is denoted by $F(S)$. We have the following properties of a nonexpansive mapping, see \cite{GK1990} for more details.
\begin{lemma}\label{lem.demiclose}
Assume that $S:C\to H$ is a nonexpansive mapping. If $S$ has a fixed point, then 
\begin{enumerate}
\item [$\rm i.$] $F(S)$ is closed convex subset of $C$.
\item [$\rm ii.$] $I-S$ is demiclosed, i.e., whenever $\left\{x_n\right\}$ is a sequence in $C$ weakly converging to some $x\in C$ and the 
sequence $\left\{(I-S)x_n\right\}$ strongly converges to some $y$ , it follows that $(I-S)x=y$.
\end{enumerate}
\end{lemma}
%%%%%%%%%%%%%%%%%%%%%%%%%%%%%%
Next, we present some concepts of the monotonicity of a bifunction and an operator (see, for instance \cite{BO1994,DHM2014,MO1992}).
\begin{definition} A bifunction $f:C\times C\to \Re$ is said to be
\begin{itemize}
\item [$\rm i.$] strongly monotone on $C$ if there exists a constant $\gamma>0$ such that
$$ f(x,y)+f(y,x)\le -\gamma ||x-y||^2,~\forall x,y\in C; $$
\item [$\rm ii.$] monotone on $C$ if 
$$ f(x,y)+f(y,x)\le 0,~\forall x,y\in C; $$
\item [$\rm iii.$] pseudomonotone on $C$ if 
$$ f(x,y)\ge 0 \Longrightarrow f(y,x)\le 0,~\forall x,y\in C;$$
\item [$\rm iv.$] Lipschitz-type continuous on $C$ if there exist two positive constants $c_1,c_2$ such that
$$ f(x,y) + f(y,z) \geq f(x,z) - c_1||x-y||^2 - c_2||y-z||^2, ~ \forall x,y,z \in C.$$
\end{itemize}
\end{definition}
From the definitions above, it is clear that  $\rm i.\Longrightarrow ii. \Longrightarrow iii.$ 
\begin{definition}
An operator $A:C \to H$ is said to be
\begin{itemize}
\item [$\rm i.$] monotone on $C$ if $\left\langle A(x)-A(y),x-y\right\rangle\ge 0$, for all $x,y\in C$;
\item [$\rm ii.$] pseudomonotone on $C$ if
$$\left\langle A(x)-A(y),x-y\right\rangle\ge 0\Longrightarrow \left\langle A(y)-A(x),y-x\right\rangle\le 0$$
for all $x,y\in C$;
\item [$\rm iii.$] $\alpha$ - inverse strongly monotone on $C$ if there exists a positive constant $\alpha$ such that
$$
\left\langle A(x)-A(y), x-y\right\rangle\ge \alpha||A(x)-A(y)||^2, \quad \forall x,y\in C;
$$
% \item [$\rm iv.$] maximal monotone on $C$ if it is monotone and its graph
% $$ G(A):=\left\{(x,A(x)):x\in C\right\} $$
%  is not a proper subset of one of any other monotone mapping;
\item [$\rm iv.$] $L$ - Lipschitz continuous on $C$ if there exists a positive constant $L$ such that $||A(x)-A(y)||\le L||x-y||$ for all $x,y\in C$. 
\end{itemize}
\end{definition}
For solving EP $(\ref{EP})$, we assume that the 
bifunction $f$ satisfies the following conditions:
\begin{itemize}
\item[(A1).] $f$ is monotone on $C$ and $f(x,x)=0$ for all $x\in C$;
\item [(A2).]  $f$ is Lipschitz-type continuous on $C$;
\item [(A3).]   $f$ is weakly continuous on $C\times C$;
\item [(A4).]  $f(x,.)$ is convex and subdifferentiable on $C$  for every fixed $x\in C.$
\end{itemize}  
% The hypothesis (A2) was introduced by Mastroeni \cite{M2003}. It is necessary to imply the convergence of  the auxiliary principle method for 
% EPs. Some examples for bifunctions which satisfy the condition $\rm (A2)$ can be found, for instance, in \cite{QMH2008}. If the bifunction 
% $f(x,y)=\left\langle A(x),y-x\right\rangle$ where $A:C\to H$ is a $L$-Lipschitz continuous operator then $f$ satisfies the condition $\rm (A2)$. 
% Indeed, from the $L$ - Lipschitz continuity of $A$, we have
% \begin{eqnarray*}
% f(x,y)+f(y,z)-f(x,z)&=&\left\langle A(x)-A(y),y-z\right\rangle\\
% &\ge&-||A(x)-A(y)||||y-z||\\
% &\ge&-L||x-y||||y-z||\\
% &\ge&-\frac{L}{2}||x-y||^2-\frac{L}{2}||y-z||^2.
% \end{eqnarray*}
% This implies that $f$ satisfies the condition $\rm (A2)$ with $c_1=c_2=L/2$. Note that, from the assumption $\rm (A2)$ with $x=z$ we obtain
% $$ 
% f(x,y)+f(y,x)\ge -(c_1+c_2)||x-y||^2,~\forall x,y\in C.
% $$
% This does not imply the monotonicity, even pseudomonotonicity, of the bifunction $f$. 
It is easy to show that under the 
assumptions $\rm (A1),(A3),(A4)$, the solution set $EP(f,C)$ of EP $(\ref{EP})$  is closed and convex 
(see, for instance \cite{QMH2008}). Therefore, from Lemma $\ref{lem.demiclose}$, the solution set $EP(f,C)\cap F(S)$ is closed and convex. 
In this paper, we assume that $EP(f,C)\cap F(S)$ is nonempty.

The metric projection $P_C:H\to C$ is defined by $P_C x=\arg\min\left\{\left\|y-x\right\|:y\in C\right\}$. Since $C$ is nonempty, closed and convex, 
$P_Cx$ exists and is unique. It is also known that $P_C$ has the following characteristic properties, 
see \cite{GR1984} for more details.
\begin{lemma}\label{lem.PropertyPC}
Let $P_C:H\to C$ be the metric projection from $H$ onto $C$. Then
\begin{itemize}
% \item [$(i)$] $P_C$ is firmly nonexpansive, i.e.,
% \begin{equation*}\label{eq:FirmlyNonexpOfPC}
% \left\langle P_C x-P_C y,x-y \right\rangle \ge \left\|P_C x-P_C y\right\|^2,~\forall x,y\in H.
% \end{equation*}
\item [$\rm i.$] For all $x\in C, y\in H$,
\begin{equation}\label{eq:ProperOfPC}
\left\|x-P_C y\right\|^2+\left\|P_C y-y\right\|^2\le \left\|x-y\right\|^2.
\end{equation}
\item [$\rm ii.$] $z=P_C x$ if and only if 
\begin{equation}\label{eq:EquivalentPC}
\left\langle x-z,z-y \right\rangle \ge 0,\quad \forall y\in C.
\end{equation}
\end{itemize}
\end{lemma}
%%%%%%%%%%%%%%%%%%%%%%%%%%%%%%%%%%%%%%%
The subdifferential of a function $g:C\to \Re$ at $x$ is defined by
\begin{equation*}\label{eq:Subdifferential}
\partial g(x)=\left\{ w\in H: g(y)-g(x)\ge \left\langle w, y-x\right\rangle,~\forall y\in C\right\}.
\end{equation*}
We recall that the normal cone of $C$ at $x\in C$ is defined by
\begin{equation*}\label{eq:NormalCone}
N_C(x)=\left\{w\in H:\left\langle w,y-x  \right\rangle \le 0,~\forall y\in C\right\}.
\end{equation*}
%%%%%%%%%%%%%%%%%%%%%%%%%%%%%
\begin{definition}[Weakly lower semicontinuity]
A function $\varphi: H\to \Re$ is called weakly lower semicontinuous at $x\in H$ if for any sequence $\left\{x_n\right\}$ in $H$ converges weakly to $x$ then 
$$\varphi(x)\le \lim\inf_{n\to\infty}\varphi(x_n). $$
\end{definition}
It is well-known that the functional $\varphi(x):=||x||^2$ is convex and weakly lower semicontinuous. 
Any Hilbert space has the Kadec-Klee property (see, for instance \cite{GK1990}), i.e., if $\left\{x_n\right\}$ is a sequence in $H$ such that $x_n\rightharpoonup x$ 
and $||x_n||\to ||x||$ then $x_n\to x$ as $n\to\infty$.
 
%%%%%%%%%%%%%%%%%%%%%%%%%%%%%%%%%%%%%%%%%
We need the following lemmas for analyzing the convergence of Algorithm $\ref{H2015}$.
%%%%%%%%%%%%%%%%%%%%%%%%%%%%%%%%%%%%
\begin{lemma}\label{lem.Equivalent_MinPro}
Let $C$ be a convex subset of a real Hilbert space H and $g:C\to \Re$ be a convex and subdifferentiable function on $C$. Then, 
$x^*$ is a solution to the following convex optimization problem
\begin{equation*}\min\left\{g(x):x\in C\right\}
\end{equation*}
if and only if  ~  $0\in \partial g(x^*)+N_C(x^*)$, where $\partial g(.)$ denotes the subdifferential of $g$ and $N_C(x^*)$ is the normal cone 
of  $C$ at $x^*$.
\end{lemma}
\begin{proof}
This is an infinite version of Theorem 27.4 in \cite{R1970} and is similarly proved.
\end{proof}
%%%%%%%%%%%%%%%%%%%%%%%%%%%%
\begin{lemma}\cite{MS2015}\label{lem.technique}
Let $\left\{M_n\right\}$, $\left\{N_n\right\}$, $\left\{P_n\right\}$ be nonnegative real sequences, 
$\alpha,\beta \in \Re$ and for all $n\ge 0$ the following inequality holds
$$ M_n\le N_n+\beta P_n-\alpha P_{n+1}. $$
If $\sum_{n=0}^\infty N_n <+\infty$ and $\alpha>\beta\ge 0$ then $\lim_{n\to\infty}M_n=0$.
\end{lemma}
%%%%%%%%%%%%%%%%%%%%%%%%%%%%%%%%%%%%%%%%%%%
%%%%%%%%%%%%%%%%%%%%%%%%%%%%%%%%%%%%%%%%%%%
\section{Convergence analysis}\label{main}
\setcounter{theorem}{0}
\setcounter{remark}{0}
\setcounter{corollary}{0}
\setcounter{algorithm}{0}
In this section, we rewrite our algorithm for more details and analyze its convergence.
\begin{algorithm}\label{algor1} \textbf{Initialization.} Chose $x_0=x_1 \in H, ~y_0=y_1\in C$ and set $C_0=Q_0=H$. 
The parameters $\lambda,k$ and $\left\{\alpha_n\right\}$ satisfy the following conditions
\begin{itemize}
\item [$\rm a. $] $0< \lambda <\frac{1}{2(c_1+c_2)},~ k>\frac{1}{1-2\lambda(c_1+c_2)}$.
\item [$\rm b. $] $\alpha_n\in [0,a]$ for some $a\in (0,1)$.
\end{itemize}
\textbf{Step 1.} Solve a strongly convex optimization program
$$
\left \{
\begin{array}{ll}
y_{n+1}=  \underset{y\in C}{\rm argmin} \{ \lambda f(y_n, y) +\frac{1}{2}||x_n-y||^2\},\\
z_{n+1}=\alpha_n y_{n+1}+(1-\alpha_n)Sy_{n+1}.\\
\end{array}
\right.
$$
\textbf{Step 2.} Compute $x_{n+1}=P_{\Omega_n}(x_0),$
where $\Omega_n=C_n\cap Q_n$,
\begin{eqnarray*}
&&C_n=\left\{z\in H: ||w_{n+1} - z||^2\leq ||x_n-z||^2+\epsilon_n \right\},\\
&&Q_n=\left\{z\in H: \left\langle x_0-x_n,z-x_n\right\rangle\le 0\right\},
\end{eqnarray*}
$\epsilon_n=k||x_n-x_{n-1}||^2+2\lambda c_2||y_n-y_{n-1}||^2-(1-\frac{1}{k}-2\lambda c_1)||y_{n+1}-y_{n}||^2$ and 
$w_{n+1}=\underset{t=y_{n+1},z_{n+1}}{\rm \arg\max} \left\{||t-x_n||\right\}$. 
Set $n:=n+1$ and go back \textbf{Step 1.}
\end{algorithm}
%%%%%%%%%%%%%%%%%%%%%%%%%%
\begin{remark}
In fact, $w_{n+1}$ in Algorithm \ref{algor1} equals to either $y_{n+1}$ or $z_{n+1}$, i.e., $w_{n+1}=y_{n+1}$ if $||y_{n+1}-x_n||\ge ||z_{n+1}-x_n||$ and $w_{n+1}=z_{n+1}$ otherwise.
\end{remark}
%%%%%%%%%%%%%%%%%%%%%%%%%%%
%%%%%%%%%%%%%%%%%%%%%%%%%%%%%%%%%%%%%
\begin{lemma}\label{lem0}
Let $\left\{x_n\right\},\left\{y_{n}\right\}$ be the sequences generated by Algorithm $\ref{algor1}$.Then, there holds the following relation for all $n\ge 0$
$$\left\langle y_{n+1}-x_n,y-y_{n+1}\right\rangle\ge \lambda\left(f(y_n,y_{n+1})-f(y_n,y)\right),~\forall y\in C.$$
\end{lemma}
%%%%%%%%%%%%%%%%%%%%%%%%%%%%%%%%%%%%%
\begin{proof}
Lemma $\ref{lem.Equivalent_MinPro}$ and the definition of $y_{n+1}$ imply that
\begin{equation*}
0\in \partial_2\left(\lambda f(y_n,y)+\frac{1}{2}||x_n-y||^2\right)(y_{n+1})+N_C(y_{n+1}).
\end{equation*}
Hence, there exist $w\in \partial_2 f(y_n,y_{n+1}):=\partial f(y_n,.)(y_{n+1})$ and $\bar{w}\in N_C(y_{n+1})$ such that
$
\lambda w+y_{n+1}-x_n+\bar{w}=0.
$
Thus, for all $y\in C$, we have
\begin{align*}
\left\langle y_{n+1}-x_n,y-y_{n+1}\right\rangle&=\lambda \left\langle w,y_{n+1}-y\right\rangle+\left\langle \bar{w},y_{n+1}-y\right\rangle\\
&\ge\lambda \left\langle w,y_{n+1}-y\right\rangle
\end{align*}
because of the definition of $N_C$. By $w\in \partial_2 f(y_n,y_{n+1})$,
\begin{equation*}
f(y_n,y)-f(y_n,y_{n+1})\ge \left\langle w,y-y_{n+1}\right\rangle,~\forall y\in C.
\end{equation*}
From the last two inequalities, we obtain the desired conclusion. 
% \begin{equation}\label{eq:1}
% \left\langle y_{n+1}-x_n,y-y_{n+1}\right\rangle\ge \lambda\left(f(y_n,y_{n+1})-f(y_n,y)\right),~\forall y\in C.
% \end{equation}
Lemma $\ref{lem0}$ is proved.
\end{proof}
%%%%%%%%%%%%%%%%%%%%%%%%%%%%%%%%%%%%%
The following lemma plays an important role in proving the convergence of Algorithm $\ref{algor1}$.
%%%%%%%%%%%%%%%%%%%%%%%%%%%%%%%%%%%%%
\begin{lemma}\label{lem1}
Assume that $x^*\in EP(f,C)\cap F(S)$. Let $\left\{x_n\right\},\left\{w_{n}\right\}$ be the sequences generated by Algorithm $\ref{algor1}$. 
Then, there holds the following relation for all $n\ge 0$
%$$ ||z_{n+1} - x^*||^2\leq||y_{n+1}-x^*||^2\leq ||x_n-x^*||^2+\epsilon_n,$$
$$||w_{n+1}-x^*||^2\leq ||x_n-x^*||^2+\epsilon_n.$$
\end{lemma}
%%%%%%%%%%%%%%%%%%%%%%%%%%%
%%%%%%%%%%%%%%%%%%%%%%%%
%%%%%%%%%%%%%%%%%%%%%%%%%%%
\begin{proof}
From $x^*\in F(S)$, the definition of $z_{n+1}$, the convexity of $||.||^2$ and the nonexpansiveness of $S$, 
\begin{align}
||z_{n+1}-x^*||^2&=||\alpha_n(y_{n+1}-x^*)+(1-\alpha_n)(Sy_{n+1}-x^*)||^2\notag\\ 
&\le\alpha_n|| y_{n+1}-x^*||^2+(1-\alpha_n)||Sy_{n+1}-Sx^*||^2\notag\\
&\le\alpha_n|| y_{n+1}-x^*||^2+(1-\alpha_n)||y_{n+1}-x^*||^2\notag\\
&=|| y_{n+1}-x^*||^2.\label{eq:1*}
\end{align}
From Lemma $\ref{lem0}$, by replacing $n+1$ by $n$, we have
\begin{equation}\label{eq:2}
\left\langle y_{n}-x_{n-1},y-y_{n}\right\rangle\ge \lambda\left(f(y_{n-1},y_{n})-f(y_{n-1},y)\right),~\forall y\in C.
\end{equation}
Substituting $y=y_{n+1}$ onto $(\ref{eq:2})$ and a straightforward computation yield
\begin{equation}\label{eq:3}
\lambda\left(f(y_{n-1},y_{n+1})-f(y_{n-1},y_{n})\right)\ge\left\langle y_{n}-x_{n-1},y_{n}-y_{n+1}\right\rangle .
\end{equation}
Lemma $\ref{lem0}$ with $y=x^*$ leads to
\begin{equation}\label{eq:3*}
\left\langle y_{n+1}-x_n,x^*-y_{n+1}\right\rangle\ge \lambda\left(f(y_n,y_{n+1})-f(y_n,x^*)\right).
\end{equation}
Since $x^*\in EP(f,C)$ and $y_n\in C$, $f(x^*,y_n)\ge 0$. Thus, from the monotonicity of $f$ one has $f(y_n,x^*)\le 0$. 
This together with $(\ref{eq:3*})$ implies that 
\begin{equation}\label{eq:4}
\left\langle y_{n+1}-x_n,x^*-y_{n+1}\right\rangle\ge \lambda f(y_n,y_{n+1}).
\end{equation}
By the Lipschitz-type continuity of $f$,
\begin{equation*}
f(y_{n-1},y_n)+f(y_n,y_{n+1})\ge f(y_{n-1},y_{n+1})-c_1||y_{n-1}-y_n||^2-c_2||y_n-y_{n+1}||^2.
\end{equation*}
Thus,
\begin{equation}\label{eq:5}
f(y_n,y_{n+1})\ge f(y_{n-1},y_{n+1})-f(y_{n-1},y_n)-c_1||y_{n-1}-y_n||^2-c_2||y_n-y_{n+1}||^2.
\end{equation}
The relations $(\ref{eq:4})$ and $(\ref{eq:5})$ lead to
\begin{eqnarray*}
\left\langle y_{n+1}-x_n,x^*-y_{n+1}\right\rangle&\ge& \lambda\left\{f(y_{n-1},y_{n+1})-f(y_{n-1},y_n)\right\}\nonumber\\
&&-\lambda c_1||y_{n-1}-y_n||^2-\lambda c_2||y_n-y_{n+1}||^2.\label{eq:6}
\end{eqnarray*}
Combining this and the relation $(\ref{eq:3})$, one gets
\begin{eqnarray*}
\left\langle y_{n+1}-x_n,x^*-y_{n+1}\right\rangle&\ge& \left\langle y_{n}-x_{n-1},y_{n}-y_{n+1}\right\rangle-\lambda c_1||y_{n-1}-y_n||^2\nonumber\\
&&-\lambda c_2||y_n-y_{n+1}||^2.
\end{eqnarray*}
Thus,
\begin{eqnarray}\label{eq:7}
2\left\langle y_{n+1}-x_n,x^*-y_{n+1}\right\rangle&-& 2\left\langle y_{n}-x_{n-1},y_{n}-y_{n+1}\right\rangle\ge-2\lambda c_1||y_{n-1}-y_n||^2\nonumber\\
&&-2\lambda c_2||y_n-y_{n+1}||^2.
\end{eqnarray}
We have the following fact
\begin{eqnarray}
&&2\left\langle y_{n+1}-x_n,x^*-y_{n+1}\right\rangle=||x_n-x^*||^2-||y_{n+1}-x^*||^2-||x_n-y_{n+1}||^2\nonumber\\ 
&=& ||x_n-x^*||^2-||y_{n+1}-x^*||^2-||x_n-x_{n-1}||^2-2\left\langle x_n-x_{n-1},x_{n-1}-y_{n+1}\right\rangle\nonumber\\
&&-||x_{n-1}-y_{n+1}||^2\nonumber\\
&=&||x_n-x^*||^2-||y_{n+1}-x^*||^2-||x_n-x_{n-1}||^2-2\left\langle x_n-x_{n-1},x_{n-1}-y_{n+1}\right\rangle\nonumber\\
&&-||x_{n-1}-y_{n}||^2-2\left\langle x_{n-1}-y_{n},y_{n}-y_{n+1}\right\rangle-||y_{n}-y_{n+1}||^2.\label{eq:8**}
\end{eqnarray}
By the triangle, Cauchy-Schwarz and Cauchy inequalities,
\begin{eqnarray*}
&-&2\left\langle x_n-x_{n-1},x_{n-1}-y_{n+1}\right\rangle\le2||x_n-x_{n-1}||||x_{n-1}-y_{n+1}||\\
&&\le2||x_n-x_{n-1}||||x_{n-1}-y_{n}||+2||x_n-x_{n-1}||||y_{n}-y_{n+1}||\\ 
&&\le|| x_n-x_{n-1}||^2+||x_{n-1}-y_{n}||^2+k|| x_n-x_{n-1}||^2+\frac{1}{k}||y_{n}-y_{n+1}||^2.
\end{eqnarray*}
This together with $(\ref{eq:8**})$ implies that 
\begin{align*}
&2\left\langle y_{n+1}-x_n,x^*-y_{n+1}\right\rangle\le ||x_n-x^*||^2-||y_{n+1}-x^*||^2+k|| x_n-x_{n-1}||^2\nonumber\\
&+2\left\langle y_n-x_{n-1},y_{n}-y_{n+1}\right\rangle+\left(\frac{1}{k}-1\right)||y_{n}-y_{n+1}||^2\label{eq:8***}.
\end{align*}
Thus, 
\begin{eqnarray}
&&2\left\langle y_{n+1}-x_n,x^*-y_{n+1}\right\rangle-2\left\langle y_n-x_{n-1},y_{n}-y_{n+1}\right\rangle\le||x_n-x^*||^2\nonumber\\
&&-||y_{n+1}-x^*||^2+k|| x_n-x_{n-1}||^2+\left(\frac{1}{k}-1\right)||y_{n}-y_{n+1}||^2\label{eq:8***}.
\end{eqnarray}
Combining $(\ref{eq:7})$ and $(\ref{eq:8***})$ we obtain 
\begin{eqnarray*}
-2\lambda c_1||y_{n-1}-y_n||^2&-&2\lambda c_2||y_n-y_{n+1}||^2\le||x_n-x^*||^2-||y_{n+1}-x^*||^2\\
&+&k|| x_n-x_{n-1}||^2+\left(\frac{1}{k}-1\right)||y_{n}-y_{n+1}||^2. 
\end{eqnarray*}
Thus, from the definition of $\epsilon_n$,
\begin{eqnarray}
||y_{n+1}-x^*||^2&\le&||x_n-x^*||^2+k|| x_n-x_{n-1}||^2\nonumber\\
&+&2\lambda c_1||y_{n-1}-y_n||^2-\left(1-\frac{1}{k}-2\lambda c_2\right)||y_{n}-y_{n+1}||^2\nonumber\\
&=&||x_n-x^*||^2+\epsilon_n.\label{eq:81*} 
\end{eqnarray}
From (\ref{eq:1*}) and (\ref{eq:81*}), we obtain
$$ ||z_{n+1}-x^*||^2\leq ||y_{n+1}-x^*||^2\leq  ||x_n-x^*||^2+\epsilon_n.$$
Thus, we obtain the desired conclusion because of the definition of $w_{n+1}$. Lemma $\ref{lem1}$ is proved.
\end{proof}
%%%%%%%%%%%%%%%%%%%%%%%
%%%%%%%%%%%%%%%%%%%%%%%%%%
%%%%%%%%%%%%%%%%%%%%%%%%%%%%%%%%%
We have the following main result.
%%%%%%%%%%%%%%%%%%%%%%%%%%%%%%%%%%%%%%%%%%%%%%%%
\begin{theorem}\label{theo.1}
Let $C$ be a nonempty closed convex subset of a real Hilbert space $H$. Assume that the bifunction $f$ satisfies all conditions $\rm (A1)-(A4)$ 
and $S:C\to C$ is a nonexpansive mapping. In addition the solution set $EP(f,C)\cap F(S)$ is nonempty. Then, the sequences 
$\left\{x_n\right\}$, $\left\{y_n\right\}$, $\left\{z_n\right\}$ generated by Algorithm 
$\ref{algor1}$ converge strongly to $P_{EP(f,C)\cap F(S)}(x_0)$.
\end{theorem}
%%%%%%%%%%%%%%%%%%%%%%%%%%%%%%%
\begin{proof} We divide the proof of Theorem $\ref{theo.1}$ into three steps.\\
\textbf{Claim 1.} $EP(f,C)\cap F(S)\subset \Omega_n$ for all $n\ge 0$.\\
Lemma $\ref{lem1}$ and the definition of $C_n$ ensure that $EP(f,C)\cap F(S)\subset C_n$ for all $n\ge 0$. 
We have $EP(f,C)\cap F(S)\subset H=\Omega_0$. Suppose that $EP(f,C)\cap F(S)\subset \Omega_n$ for some $n\ge 0$. From $x_{n+1}=P_{\Omega_n}(x_0)$ 
and Lemma $\ref{lem.PropertyPC}$.ii., we see that $\left\langle z-x_{n+1},x_0-x_{n+1}\right\rangle\le 0$ for all $z\in \Omega_n$. Thus, 
$\left\langle z-x_{n+1},x_0-x_{n+1}\right\rangle\le 0$ for all $z\in EP(f,C)\cap F(S)$. Hence, from the definition of $Q_{n+1}$, $EP(f,C)\cap F(S)\subset Q_{n+1}$ or 
$EP(f,C)\cap F(S)\subset \Omega_{n+1}$. By the induction, $EP(f,C)\cap F(S)\subset \Omega_n$ for all $n\ge 0$. Since $EP(f,C)\cap F(S)$ is nonempty, 
so $\Omega_n$ is. Therefore, $P_{\Omega_n}(x_0)$ and $P_{EP(f,C)\cap F(S)}(x_0)$ are well-defined.\\
\textbf{Claim 2.} $\lim\limits_{n\to\infty}||x_{n+1}-x_n||=\lim\limits_{n\to\infty}||z_n-x_n||=\lim\limits_{n\to\infty}||y_n-x_n||=\lim\limits_{n\to\infty}||y_{n+1} - y_{n}||=0$
and $\lim\limits_{n\to\infty}||y_n-Sy_{n}||=0$.\\
Indeed, from the definition of $Q_n$ and Lemma $\ref{lem.PropertyPC}$.ii., $x_n=P_{Q_n}(x_0)$. Thus, from Lemma $\ref{lem.PropertyPC}$.i., we have
\begin{equation}\label{eq:8}
||z-x_n||^2\le||z-x_0||^2-||x_n-x_0||^2,~\forall z\in Q_n.
\end{equation}
Substituting $z=x^\dagger:=P_{EP(f,C)\cap F(S)}(x_0)\in Q_n$ onto $(\ref{eq:8})$, one has 
\begin{equation}\label{eq:8*}
||x^\dagger-x_0||^2-||x_n-x_0||^2\ge ||x^\dagger-x_n||^2\ge 0.
\end{equation}
Thus, the sequence $\left\{||x_n-x_0||\right\}$ and $\left\{x_n\right\}$ are bounded. The relation $(\ref{eq:8})$ with $z=x_{n+1}\in Q_n$ leads to
\begin{equation}\label{eq:9}
0\le||x_{n+1}-x_n||^2\le||x_{n+1}-x_0||^2-||x_n-x_0||^2.
\end{equation}
This implies that $\left\{||x_n-x_0||\right\}$ is non-decreasing. Hence, there exists the limit of $\left\{||x_n-x_0||\right\}$. 
By $(\ref{eq:9})$, 
$$
\sum_{n=1}^K||x_{n+1}-x_n||^2\le ||x_{K+1}-x_0||^2-||x_1-x_0||^2,~\forall K\ge 1.
$$
Passing the limit in the last inequality as $K\to\infty$, we obtain
\begin{equation}\label{eq:10*}
\sum_{n=1}^\infty||x_{n+1}-x_n||^2<+\infty.
\end{equation}
Thus,
\begin{equation}\label{eq:11*}
\lim_{n\to\infty}||x_{n+1}-x_n||=0.
\end{equation}
From the definition of $C_n$ and $x_{n+1}\in C_n$,
\begin{equation}\label{eq:12}
 ||w_{n+1} - x_{n+1}||^2\leq ||x_n-x_{n+1}||^2+\epsilon_n.
\end{equation}
Set $M_{n}=||w_{n+1} - x_{n+1}||^2$, $N_n=||x_n-x_{n+1}||^2+k||x_n-x_{n-1}||^2$, $P_n=||y_n-y_{n-1}||^2$, 
$\beta=2\lambda c_2$, 
and  $\alpha=1-\frac{1}{k}-2\lambda c_1$. From the definition of $\epsilon_n$, $\epsilon_n=k||x_n-x_{n-1}||^2+\beta P_n-\alpha P_{n+1}$. 
Thus, from $(\ref{eq:12})$, 
\begin{equation}\label{eq:13*}
M_{n}\le N_n+\beta P_n-\alpha P_{n+1}.
\end{equation}
From the hypothesises of $\lambda, ~k$ and $(\ref{eq:10*})$, we see that $\alpha>\beta\ge 0$ and $\sum_{n=1}^\infty N_n<+\infty$. Lemma 
$\ref{lem.technique}$ and $(\ref{eq:13*})$ imply that $M_{n}\to 0$, or $||w_{n+1} - x_{n+1}||\to 0$. Thus, $||w_{n+1} - x_{n}||\to 0$ due 
to the relation (\ref{eq:11*}) and the triangle inequality. Hence, from the definition of $w_{n+1}$, we obtain
\begin{equation*}
\lim_{n\to\infty}||z_{n+1}- x_{n}||=\lim_{n\to\infty}||y_{n+1}- x_{n}||=0.
\end{equation*}
These together with (\ref{eq:11*}) and the triangle inequality imply that
\begin{equation}\label{eq:14}
\lim_{n\to\infty}||z_{n+1}- x_{n+1}||=\lim_{n\to\infty}||y_{n+1}- x_{n+1}||=\lim_{n\to\infty}||z_{n+1}-y_{n+1}||=0.
\end{equation}
By the triangle inequality,
\begin{eqnarray*}
||y_{n+1}-y_{n}||\le||y_{n+1}-x_{n+1}||+||x_{n+1}-x_n||+||x_{n}-y_{n}||.
\end{eqnarray*}
From the last inequality, $(\ref{eq:11*})$ and $(\ref{eq:14})$,
\begin{equation}\label{eq:15*}
\lim_{n\to\infty}||y_{n+1} - y_{n}||=0.
\end{equation}
From the definition of $z_{n+1}$ we have 
\begin{equation*}\label{eq:16}
||z_{n+1}-y_{n+1}||=(1-\alpha_n)||y_{n+1}-Sy_{n+1}||\ge (1-a)||y_{n+1}-Sy_{n+1}||.
\end{equation*}
This together with $(\ref{eq:14})$ and $1-a>0$ implies that $ \lim_{n\to\infty}||y_{n+1} - Sy_{n+1}||=0$.\\
\textbf{Claim 3.} $x_n\to x^\dagger=P_{EP(f,C)\cap F(S)}(x_0)$ as $n\to \infty$.\\
Assume that $p$ is any weak cluster point of $\left\{x_n\right\}$. 
Without loss of generality, we can write $x_n\rightharpoonup p$ as $n\to \infty$. Since $||x_n-y_n||\to 0$, $y_n\rightharpoonup p$. 
Now, we show that $p\in EP(f,C)\cap F(S)$. Indeed, from Claim 2 and the demiclosedness of $S$ we obtain $p\in F(S)$. By 
Lemma $\ref{lem0}$, we get 
\begin{equation}\label{eq:15}
\lambda\left(f(y_n,y)-f(y_n,y_{n+1})\right)\ge \left\langle x_n-y_{n+1},y-y_{n+1}\right\rangle,~\forall y\in C.
\end{equation}
Passing to the limit in $(\ref{eq:15})$ as $n\to\infty$ and using Claim 2, the bounedness of $\left\{y_n\right\}$ 
and $\lambda>0$ we obtain $f(p,y)\ge 0$ for all $y\in C$. Hence, $p\in EP(f,C)$. Thus, $p\in EP(f,C)\cap F(S)$.  From the inequality $(\ref{eq:8*})$, we get
$ ||x_n-x_0||\le ||x^\dagger-x_0||,$
where $x^\dagger=P_{EP(f,C)\cap F(S)}(x_0)$. By the weak lower semicontinuity of the norm $||.||$ and $x_n\rightharpoonup p$, we have
\begin{equation*}
||p-x_0||\le \lim_{n\to\infty}\inf||x_{n}-x_0||\le \lim_{n\to\infty}\sup||x_{n}-x_0||\le||x^\dagger-x_0||.
\end{equation*}
By the definition of $x^\dagger$, $p=x^\dagger$ and $\lim_{n\to\infty}||x_{n}-x_0||=||x^\dagger-x_0||$. Thus, $\lim_{n\to\infty}||x_{n}||=||x^\dagger||$. 
By the Kadec-Klee property of the Hilbert space $H$, we have $x_{n}\to x^\dagger=P_{EP(f,C)\cap F(S)}(x_0)$ as $n\to\infty$. From Claim 2, 
we also see that $\left\{y_n\right\}$ and $\left\{z_n\right\}$ converge strongly to $P_{EP(f,C)\cap F(S)}(x_0)$. Theorem $\ref{theo.1}$ is proved.
\end{proof}
%%%%%%%%%%%%%%%%%%%%%%%%%%%%%%%%%%
%%%%%%%%%%%%%%%%%%%%%%%%%%%%555
\begin{corollary}\label{cor1}
Let $C$ be a nonempty closed convex subset of a real Hilbert space $H$. Assume that the bifunction $f$ satisfies all conditions $\rm (A1)-(A4)$. 
In addition the solution set $EP(f,C)$ is nonempty. Let $\left\{x_n\right\},\left\{y_n\right\}$ be two sequences generated by the following manner: 
$x_0=x_1\in H$, $y_0=y_1\in C$ and
\begin{equation*}
\left \{
\begin{array}{ll}
y_{n+1}=  \underset{y\in C}{\rm argmin} \{ \lambda f(y_n, y) +\frac{1}{2}||x_n-y||^2\},\\
C_n=\left\{z\in H: ||y_{n+1}-z||^2\leq ||x_n-z||^2+\epsilon_n \right\},\\
Q_n=\left\{z\in H: \left\langle x_0-x_n,z-x_n\right\rangle\le 0\right\},\\
x_{n+1}=P_{C_n\cap Q_n}(x_0),
\end{array}
\right.
\end{equation*}
where $\epsilon_n, \lambda, k$ are defined as in Algorithm $\ref{algor1}$. Then, the sequences $\left\{x_n\right\}$, $\left\{y_n\right\}$ 
converge strongly to $P_{EP(f,C)}(x_0)$.
\end{corollary}
%%%%%%%%%%%%%%%%%%%%%%%%%%%
%%%%%%%%%%%%%%%%%%%%%%%%%%%%%%%
\begin{proof}
Set $S=I$ (the identity operator in $H$), from Algorithm $\ref{algor1}$ we have $z_{n+1}=y_{n+1}=w_{n+1}$. Thus, Corollary $\ref{cor1}$ is directly followed from Theorem $\ref{theo.1}$.
\end{proof}
%%%%%%%%%%%%%%%%%%%%%%%%%%%%%%%
\begin{corollary}\label{cor2}
Let $C$ be a nonempty closed convex subset of a real Hilbert space $H$. Assume that $A:C\to H$ is a monotone and $L$-Lipschitz continuous operator 
and $S:C\to C$ is a nonexpansive mapping such that the solution set $VI(A,C)\cap F(S)$ is nonempty, where $VI(A,C)$ is the solution set of VIP 
(\ref{VIP}). Let $\left\{x_n\right\},\left\{y_n\right\}, \left\{z_n\right\}$ 
be the sequences generated by the following manner: $x_0=x_1\in H$, $y_0=y_1\in C$ and
\begin{equation}
\left \{
\begin{array}{ll}
y_{n+1}=  P_C(x_n-\lambda A(y_{n})),\\
z_{n+1}=\alpha_n y_{n+1}+(1-\alpha_n)Sy_{n+1},\\
C_n=\left\{z\in H: ||w_{n+1} - z||\leq ||x_n-z||^2+\epsilon_n\right\},\\
Q_n=\left\{z\in H: \left\langle x_0-x_n,z-x_n\right\rangle\le 0\right\},\\
x_{n+1}=P_{C_n\cap Q_n}(x_0),
\end{array}
\right.
\end{equation}
where $w_{n+1},\epsilon_n, \lambda, k$ are defined as in Algorithm $\ref{algor1}$ with $c_1=c_2=L/2$. Then, the sequences 
$\left\{x_n\right\}$, $\left\{y_n\right\},\left\{z_n\right\}$ 
converge strongly to $P_{VI(A,C)\cap F(S)}(x_0)$.
\end{corollary}
%%%%%%%%%%%%%%%%%%%%%%%%%%%%
\begin{proof}
Set $f(x,y)=\left\langle A(x),y-x\right\rangle$ for all $x,y\in C$. It is clear that $f$ satisfies the conditions $\rm (A1)$, $\rm (A3)$, $\rm (A4)$ automatically. 
Now, we show that the condition $\rm (A2)$ holds for the bifunction $f$. Indeed, from the 
$L$ - Lipschitz continuity of $A$, we have
\begin{eqnarray*}
f(x,y)&+&f(y,z)-f(x,z)=\left\langle A(x)-A(y),y-z\right\rangle\ge-||A(x)-A(y)||||y-z||\\
&\ge&-L||x-y||||y-z||\ge-\frac{L}{2}||x-y||^2-\frac{L}{2}||y-z||^2.
\end{eqnarray*}
This implies that $f$ satisfies the condition $\rm (A2)$ with $c_1=c_2=L/2$. According to Algorithm $\ref{algor1}$, we have
\begin{eqnarray*}
y_{n+1}&=&\underset{y\in C}{\rm argmin} \{ \lambda \left\langle A(y_n),y-y_n\right\rangle+\frac{1}{2}||x_n-y||^2\}\\ 
&=&\underset{y\in C}{\rm argmin} \{\frac{1}{2}||y-(x_n-\lambda A(y_n))||^2-\frac{\lambda^2}{2}||A(y_n)||^2-\lambda \left\langle A(y_n),y_n-x_n\right\rangle\}\\
&=&\underset{y\in C}{\rm argmin} \{\frac{1}{2}||y-(x_n-\lambda A(y_n))||^2\}\\
&=&P_C\left(x_n-\lambda A(y_n)\right).
\end{eqnarray*}
Hence, Corollary \ref{cor1} is directly followed from Theorem \ref{theo.1}.
\end{proof}
%%%%%%%%%%%%%%%%%%%%%%%%%%%%%%%%%%%%
\begin{remark}
\begin{itemize}
\item From the proofs of Lemma $\ref{lem1}$ and Theorem $\ref{theo.1}$, we see that Theorem $\ref{theo.1}$, Corollaries \ref{cor1} and \ref{cor2} 
remain true if the monotonicity is replaced by the pseudomonotonicity.
\item Corollary \ref{cor2} can be considered as an improvement of the results in \cite{CGR2011,NT2006a} in the sense that we only need to find a projection 
onto the constrained set $C$ at each iteration.
\item The set $\Omega_n$ in Step 2 of Algorithm \ref{algor1} can be replaced by $\Omega_n=C_n^1\cap C_n^2\cap Q_n$, where $C_n^1,C_n^2$ are 
two halfspaces defined by
\begin{eqnarray*}
&&C_n^1=\left\{z\in H: ||z_{n+1} - z||^2\leq ||y_{n+1}-z||^2\right\},\\
&&C_n^2=\left\{z\in H: ||y_{n+1} - z||^2\leq ||x_n-z||^2+\epsilon_n \right\}.
\end{eqnarray*}
\end{itemize}

\end{remark}
%%%%%%%%%%%%%%%%%%%%%%%%%%%
\section{Numerical examples}\label{Numer}
%%%%%%%%%%%%%%%%%%%%%%%%%%%%%%%%%
In this section, we consider two numerical examples to illustrate the convergence of Algorithm $\ref{H2015}$. The bifunction $f:C\times C\to \Re$ 
which comes from the Nash-Cournot equilibrium model in \cite{QMH2008,SVH2013} is defined by 
$$ f(x,y)=\left\langle Px+Qy+q,y-x\right\rangle, $$
where $q\in \Re^n$, $P,~Q\in \Re^{n\times n}$ are two matrices of order $n$ such that $Q$ is symmetric, positive semidefinite and $Q-P$ is 
negative semidefinite. By \cite[Lemma 6.2]{QMH2008}, $f$ is monotone and Lipschitz-type continuous with $c_1=c_2=\frac{1}{2}||P-Q||$. 
In two numerical experiments below we chose $\lambda=\frac{1}{5c_1}, k=6, x_1=x_0\in \Re^n$ and $y_0,y_1$ are the zero vector. All 
convex quadratic optimization programs are solved by the MALAB Optimization Toolbox. The algorithm is performed on a 
PC Desktop Intel(R) Core(TM) i5-3210M CPU @ 2.50GHz 2.50 GHz, RAM 2.00 GB. 

%%%%%%%%%%%%%%%%%%%%%%%%%%%%%%%%

\noindent\textit{Example 1.} We consider the feasible set $C$ is defined by 
$$ C=\left\{x\in \Re^3:\sum_{i=1}^3x_i\ge 1,~0\le x_i\le 1,~i=1,\ldots,3\right\} $$
and $S$ is the identity operator $I$. This example is tested with $q=(1;-2;3)^T$ and
$$
    P =
    \left(\begin{array}{ccc}
     3.1&2&0\\
     2&3.6&0\\
0&0&3.5\\
    \end{array}\right),~
Q =
    \left(\begin{array}{ccc}
     1.6&1&0\\
     1&1.6&0\\
0&0&1.5\\
    \end{array}\right).
$$
The iterate $x_{n+1}$ is expressed by the explicit formula in \cite{SS2000}. 
In this case, the solution set $EP(f,C)\cap F(S)=EP(f,C)$ is not known. Thus, the stopping criterion used in this experiment is 
$||w_{n+1}-x_n||\le TOL=0.0001$. Table $\ref{tab:1}$ shows 
the numbers of iterates (Iter.), time for execution of Algorithm $\ref{H2015}$ in second (CPU in sec.) and 
approximation solutions $x_n$ to EP for chosing different starting points.

%%%%%%%%%%%%%%%%%%%%%%%%%%%%%%%%
\begin{table}[ht]\caption{Results for given starting points in \textit{Example 1}.}\label{tab:1}
\medskip\begin{center}
\begin{tabular}{|c|c|c|c|}
\hline
$x_0$ & Iter. &CPU in sec.& $x_{n}$\\ \hline
(1; 3; 1)&377&9.18&(0.0000004; 0.9806232; 0.0194736)\\ \hline
(-3; 4; 1)&220&4.96&(0.0000000; 0.9806290; 0.0194844)\\ \hline
(3; -2; 1)&480&15.53&(0.0000004; 0.9806289; 0.0194885)\\ \hline
 \end{tabular}
\end{center}
\end{table}

\noindent\textit{Example 2.} We consider the constrained set $C$ as a box by
 $$ C=\left\{x\in \Re^3:0\le x_i\le 1,~i=1,\ldots,3\right\}. $$
Two matrices $P, Q$ are defined as in \textit{Example 1}, $q$ is the zero vector.
Let $C_i,i=1,2,3$ be three halfspaces such that $C\cap C_1\cap C_2\cap C_3= \emptyset$. 
For each $x\in \Re^3$, set
$$\Phi(x):=\frac{1}{3}\sum_{i=1}^3\min_{z\in C_i}||x-z||^2=\frac{1}{3}\sum_{i=1}^3d^2(x,C_i),~ \mbox{and}~ C_{\Phi}:=\underset{x\in C}{\arg\min}~\Phi(x).$$
Define $S:C\to C$ by 
$$ S:=P_C\left(\frac{1}{3}\sum_{i=1}^3P_{C_i}\right). $$
Since the projection is nonexpansive, $S$ is nonexpansive and $F(S)=C_\Phi$ (see, \cite[Proposition 4.2]{Y2001} ). 
In this example, we chose $C_1=\left\{x\in \Re^3:3x_1+2x_2+x_3\le -6\right\}$, $C_2=\left\{x\in \Re^3:5x_1+4x_2+3x_3\le -12\right\}$ 
and $C_3=\left\{x\in \Re^3:2x_1+x_2+x_3\le -4\right\}$. 
It is easy to show that $EP(f,C)\cap F(S)=\left\{0\right\}$. For each starting 
point $x_0$ then the sequence $\left\{x_n\right\}$ generated by Algorithm $\ref{H2015}$ converges strongly to 
$x^\dagger:=P_{EP(f,C)\cap F(S)}(x_0)=0$. The termination criterion is $||x_n-x^\dagger||\le TOL=0.001$.
The results are shown in Table $\ref{tab:2}$ for chosing different starting points and parameters $\alpha_n$.
%%%%%%%%%%%%%%%%%%%%%%%%%%%%%%%%

\begin{table}[ht]\caption{Results for different starting points and parameters in \textit{Example 2}.}\label{tab:2}
\medskip\begin{center}
\begin{tabular}{|c|c|c|c|c|c|c|}
\hline
 $x_0$&\multicolumn{2}{c|}{$\alpha_n=\frac{n-1}{2(n+1)}$} &\multicolumn{2}{c|}{$\alpha_n=10^{-n}$}&\multicolumn{2}{c|}{$\alpha_n=\frac{1}{\log_{10}(n+1)}$}
\\ \cline{2-7}
& Iter. &CPU in sec.& Iter. &CPU in sec.& Iter. &CPU in sec.\\ \hline
(1; 3; 1)&23&1.54&26&1.75&67&4.45\\ \hline
(-3; 4; 1)&47&3.07&20&1.34&81&5.46\\ \hline
(3; -2; 1)&29&1.96&17&1.03&47&2.40\\ \hline
(-2; 3; -1)&7&0.40&6&0.26&18&1.18\\ \hline
 \end{tabular}
\end{center}
\end{table}
%%%%%%%%%%%%%%%%%%%%%%%%%%%%%%%%%%%%%

%%%%%%%%%%%%%%%%%%%%%%%%%%%%%%%%%%%%%
The study of the numerical experiments here is preliminary and it is obvious that EPs and fixed point problems depend on the structure of the constrained set $C$, 
 the bifunction $f$, and the mapping $S$. However, the results in Tables $\ref{tab:1}$ and $\ref{tab:2}$ have illustrated the convergence of our proposed algorithm 
and we also see that the number of iterative step and time for execution of the algorithm depend on the starting point $x_0$ and the parameter $\alpha_n$.
%%%%%%%%%%%%%%%%%%%%%%%%%%%%%
%\section*{Conclusions} The paper proposes a novel hybrid algorithm for EPs and fixed point problems
%%%%%%%%%%%%%%%%%%%%%%%%%%%%%

\end{document}